\documentclass[11pt,a4paper]{article}
\usepackage{amsfonts}

\usepackage{float}
\usepackage{graphicx,setspace}
\usepackage{amsfonts,amsmath, amssymb}
\usepackage{epstopdf}
\usepackage{color,amsmath,amssymb, amsfonts,amstext,amsthm}
\usepackage{bbm}
\usepackage{mathrsfs}
\usepackage{cite}
\usepackage{amsmath}
\usepackage{appendix}

\DeclareMathOperator{\dive}{div}

\newtheorem{theorem}{\textbf{\ \ \quad Theorem}}[section]

\newtheorem{remark}{\textbf{\ \ \quad Remark}}[section]
\newtheorem{col}{\textbf{\ \ \quad Corollary}}[section]

\renewcommand{\l}{\left}
\renewcommand{\r}{\right}

\newcommand{\lam}{\lambda}

\renewcommand{\phi}{\varphi}

\renewcommand{\a}{\alpha}

\newcommand{\delt}{\bigtriangleup t}

\newcommand{\be}{\begin{equation}}
\newcommand{\ee}{\end{equation}}
\newcommand\bes{\begin{eqnarray}}
\newcommand\ees{\end{eqnarray}}
\newcommand{\bess}{\begin{eqnarray*}}
\newcommand{\eess}{\end{eqnarray*}}

\newcommand{\dx}{{\rm d}x}
\newcommand{\dy}{{\rm d}y}

\newcommand{\befig}{\begin{figure}}
\newcommand{\enfig}{\end{figure}}

\newcommand{\bear}{\begin{eqnarray}}
\newcommand{\enar}{\end{eqnarray}}

\newcommand{\bearn}{\begin{eqnarray*}}
\newcommand{\enarn}{\end{eqnarray*}}
 \def\theequation{\arabic{section}.\arabic{equation}}
{\setlength\arraycolsep{2pt}}

\title{ 
Dynamical behavior of a nonlocal Fokker-Planck equation for a stochastic system with tempered stable noise
}
\author{Li Lin $^1$, Jinqiao Duan $^2$,  Xiao Wang$^3$  and Yanjie Zhang$^4$  \\
\\
\ \\
{\small \it $^1$ Center for Mathematical Sciences, Huazhong University of Science and Technology }\\
  {\small \it Wuhan 430074, China  }\\
{\small \tt email: linli@hust.edu.cn}\\
{\small \it $^2$ Department of Applied Mathematics, Illinois Institute of Technology }\\
  {\small \it Chicago, IL 60616, USA }\\
   {\small \tt email:duan@iit.edu }\\
{\small \it $^3$ School of Mathematics and Statistics, Henan University }\\
  {\small \it Kaifeng 475001, China  }\\
{\small \tt email: xwang@vip.henu.edu.cn}\\
{\small \it $^4$ School of Mathematics, South China University of Technology }\\
  {\small \it Guangzhou 510000,  China }\\
  {\small \tt email:zhangyj18@scut.edu.cn}\\
}
\begin{document}
\maketitle

\begin{abstract}
 We characterize a stochastic dynamical system with tempered stable noise, by examining its probability density evolution.  This probability density function satisfies a nonlocal Fokker-Planck equation.  First, we prove a superposition principle that the probability measure-valued solution to this nonlocal Fokker-Planck equation is equivalent to the martingale solution composed with the inverse stochastic flow. This result together with a Schauder estimate leads to the existence and uniqueness of strong solution for the nonlocal Fokker-Planck equation. Second, we devise a convergent  finite difference method to simulate the probability density function by solving the nonlocal Fokker-Planck equation.
Finally, we apply our  aforementioned theoretical and numerical results to a nonlinear filtering system by simulating a nonlocal Zakai equation.

\end{abstract}

\baselineskip=15pt

{\bfseries  Stochastic dynamic systems are widely used to describe various complex phenomena. The random fluctuations in complex phenomena are usually non-Gaussian. How to capture the uncertainly propagates and evolves for the nonlinear stochastic dynamical systems is an important issue. A popular method is to obtain the probability density function of the solution paths, which contains the complete statistical information.  In this paper, we prove a superposition principle that the probability measure-valued solution to this nonlocal Fokker-Planck equation is equivalent to the martingale solution. Meanwhile, we devise a convergent  finite difference method to simulate the probability density function by solving the nonlocal Fokker-Planck equation, which helps  us to detect the dynamic behavior of the stochastic dynamic system. Further, we apply our aforementioned theoretical
and numerical results to a nonlinear filtering system by simulating a nonlocal Zakai equation. The results established in this paper can be used to examine dynamical behaviors for financial markets, climate dynamics and physics.}

\section{Introduction}
\par
 L\'evy processes have been observed experimentally in fluid dynamics and polymers and have been used to describe subrecoil laser cooling, turbulent fluids, very stiff polymers, and the spectral random walk of a single molecule embedded in a solid \cite{Bli, Ts13, Dit, Ale, Taqqu}.  However, in physical systems, the variance of any stationary processes is finite. Then an unavoidable cutoff is always present. For example, in the case of a single molecule embedded in a solid, due to the minimal length between the molecule and the nearest two-level systems, a cutoff is present in the distribution of the jumps of the resonance frequency \cite{Ros, Kuchler}. Many scholars have introduced a smooth exponential regression towards zero. This makes it possible to derive an analytic expression for characteristic function and enables one to replace simulations by more straightforward calculations. This is the tempered stable processes what we mention in our paper. Unlike the $\a$-stable counterpart, a tempered stable L\'evy process   has mean, variance, and moments of all order \cite{Con}.


 The Fokker-Planck equation (FPE) is an important deterministic tool for quantifying stochastic differential equation. The solution of FPE describes the evolution of the transition probability density for  a stochastic system. The nonlocal FPE for systems with non-Gaussian noise have been derived (e.g.,\cite{ds,Zhang}).
 It is hardly possible to have analytical solution for the nonlocal FPE except the special drift term \cite{Xu3}. A few authors have considered numerical simulations of such nonlocal FPE. Cont and  Voltchkova \cite{Cont05}  presented a finite difference method for solving parabolic  integro-differential equations with possibly singular kernels, which can be used to price European and barrier options via stochastic models with tempered L\'evy processes.  Li and Deng \cite{Li} developed a high order difference scheme for a  tempered fractional diffusion equation on a bounded domain, together with stability and convergence analysis. Gao {\it et al.}~\cite{Gao} developed a fast finite difference scheme to simulate the nonlocal FPE on either a bounded or infinite domain. Xu {\it et al.}~\cite{Xu1, Xu2} developed the path integral method to solve one-dimensional space fractional Fokker-Planck-Kolmogorov equations.

Our study is divided into two  parts. In the first part, we study the superposition principle and global parabolic Schauder estimate for the nonlocal FPE under the natural condition $\alpha+\beta > 1$.   The operator $\mathcal{L}$ is not a stable-like operator, i.e., we can not find two positive constants $c_1$ and $c_2$ such that
\begin{equation}
c_1\leq \frac{1}{e^{\lambda |y|}}\leq c_2, ~~ \forall y \in \mathbb{R}^{d}.
\end{equation}
So we can not apply the recent works of Zhang  et al. \cite{Zhao} and Chen  et al. \cite{Chen}  who address the parabolic Dirichlet problem through using a probabilistic argument. A natural question  is whether  the global parabolic Schauder estimate
$\|p(x,t)\|_{L^\infty\l([0,T],C^{\a+\beta}_b\r)}\leq C\|g\|_{C^{1+\beta}_b}$ holds. We will answer this in Section $2$. Further,
  we construct a finite difference scheme to simulate the nonlocal Fokker-Planck equation. Under a specified condition, the semi-discrete scheme is shown to satisfy the discrete maximum principle and to be convergent. In the second part, we apply the aforementioned results to obtain the strong form of Zakai equation and conduct its numerical similation.
\medskip

%
\par
This paper is organized as follows. In section 2, we study the well-poseness for the nonlocal FPE associated with stochastic  differential equation driven by tempered stable noise. Further, we construct a convergent finite difference scheme to simulate the nonlocal FPE, and a numerical experiment is conducted to confirm the theoretical results. In section 3, we apply the aforementioned results to  nonlinear filtering problem. Some concluding remarks are made in Section 4.

\section{METHODOLOGY}
{\bf{A: The nonlocal FPE}}\\
In this paper, we consider the following  stochastic differential equation (SDE)
\begin{equation}
\label{SDE01}
\mathrm{d}X_t=f(X_t)\mathrm{d}t+ \mathrm{d}L_t,
\end{equation}
where $f$ is a $d$-dimensional Borel measurable  function, and $L_t$ is a $d$-dimensional tempered stable L\'evy process  with triplet $(0,0, \nu)$.  This triplet means that the process has  shift   zero,  diffusion   zero,  and L\'evy jump measure $\nu$. The   jump measure $\nu$ is obtained through multiplying the
$\a$-stable L\'evy measure by an exponentially decaying function. This transform is called exponential tilting of the L\'evy measure for $\a$-stable L\'evy measure (see \cite{Deng}), i.e.,
\begin{equation}
\label{LevyMesure2}
\nu(\dy) =  \frac{C_{d,\a} \dy }{e^{\lam |y|}|y|^{d+\a}}.
\end{equation}
Here $C_{d,\a}$ is a constant,  and $\lambda$ is the positive tempering parameter (see \cite{De20}). The parameter $\a \in(0,1)\cup (1,2)$ is called the stable index.


The corresponding FPE for the SDE \eqref{SDE01} is the following nonlocal parabolic equation, i.e.,
\begin{equation}
\label{parabo}
   \frac{\partial p}{\partial t}=-\nabla \cdot\left(f(x)p(x,t)\right)+\mathcal{L}p(x,t),
\end{equation}
where
\begin{equation}
\begin{aligned}
\mathcal{L}p(x,t)&=\int_{\mathbb{R}^d\setminus \{0\}}
   \left[p(x+y,t)-p(x,t)-1_{\{|y|<1\}}(y)y \cdot \nabla p(x,t)  \right]\nu(\mathrm{d}y) \\
   &={\bf{ p.v.}} \int_{\mathbb{R}^d\setminus \{0\}}
   \left[p(x+y,t)-p(x,t)\right]\nu(\mathrm{d}y).
   \end{aligned}
\end{equation}
In the following, we are  interested in studying the dynamical properties and numerical analysis for the equation \eqref{parabo} in the principle value sense.

{\bf{B: Examing the solution of nonlocal FPE }}
\par
 Let $\mu_t(dx)=p(x,t)dx$ be the marginal law of $X_t$. By It\^o's formula, $\mu_t$ solves the following
Fokker-Planck equation in the distributional sense, i.e.,
\begin{equation}
\label{0FPE}
  \partial_t \mu_t =-\nabla\cdot\left(f(x)\mu_t\right)+ \int_{\mathbb{R}^d\setminus \{0\}}
   \left[\mu_t(x+y)-\mu_t(x)-1_{\{|y|<1\}}(y)y \cdot\frac{ \nabla\mu_t }{\partial x}\right]\nu(\mathrm{d}y).
\end{equation}

The proof of the existence and uniqueness of the equation \eqref{FPE3} is based on two ingredients. The first ingredient will be to establish the superposition principle  for probability measure-valued solution to a nonlocal FPE.  The second ingredient will then be to establish the global Schauder estimate for the solution of (\ref{FPE3}).\\
{\bf{B1: Superposition Principle}}\\
Next we will establish the superposition principle for probability measure-valued solution to the nonlocal FPE.
Throughout this paper, we make the following assumptions.\\
{\bf Hypothesis  H.1.}
The drift coefficient $f$  is locally bounded, measurable and there exists a positive constant $C$ such that
\begin{equation}
\label{assum}
\sup_{x} \langle x, f(x) \rangle ^{+}\leq C(1+|x|^2),
\end{equation}
where $\langle x, f(x) \rangle ^{+}$ denote the non-negative part of $\langle x, f(x) \rangle.$\\
{\bf Hypothesis  H.2.}
The drift coefficient $f$ is locally $\beta$-H\"{o}lder continuous, $\beta\in(0, 1)$, i.e.,  there exists a positive constant $K_0$ such that
\begin{equation}
\label{f2}
|f(x)-f(y)|\leq K_0|x-y|^\beta, ~~x, y \in \mathbb{R}^d, ~s.t~~ |x-y|\leq 1.
\end{equation}
\begin{theorem}[Superposition principle]
Under Hypothesis  $\bf{H.1}$, for every weak solution $(\mu_t)_{t\geq 0}$ of nonlocal FPE \eqref{0FPE}, there is a martingale solution $\mathbb{P}\in \mathcal{M}^{\mu_0}_{0}(\widetilde{\mathcal{L}})$ such that
\begin{equation}
\mu_t=\mathbb{P}\circ X^{-1}_t, ~~\forall t \geq 0,
\end{equation}
\begin{proof}
Using the similar technique  as \cite [Theorem 1.5]{Long}, we only need to verify
$$\hbar^{\nu}(x):=\int_{|z|>1+|x|}\log \left(1+\frac{|z|}{1+|x|}\right)\nu(dz)<\infty.$$ In fact, $\forall \lambda >0$, $\alpha \in (0,1)\cup(1,2)$ and $\beta \in (0, \alpha \wedge 1)$, there exists a positive constant $C_{\alpha}$, such that
\begin{equation}
\begin{aligned}
\hbar^{\nu}(x) &= C_{\alpha} \int_{|z|>1+|x|} \log \left(1+\frac{|z|}{1+|x|}\right)\frac{1}{e^{\lam |z|}|z|^{1+\a}}dz\\
& \leq C_{\alpha}\int_{|z|>1+|x|}\log \left(1+\frac{|z|}{1+|x|}\right)\frac{1}{|z|^{1+\a}}dz\\
& \leq C_{\alpha} \int_{|z|>1+|x|} \left(\frac{|z|}{1+|x|}\right)^{\beta}\frac{1}{|z|^{1+\a}}dz \\
& \leq C_{\alpha} \int_{|z|>1}\frac{1}{|z|^{1+\alpha-\beta}}dz\\
& \leq C_{\alpha}.
\end{aligned}
\end{equation}
\end{proof}
\end{theorem}
Next we will show the existence and uniqueness of strong solution for the nonlocal FPE \eqref{FPE3}.\\
{\bf{B2: Schauder Estimate}}\\
In this subsection, we are interested in the absolutely continuous solution of the nonlocal FPE \eqref{0FPE}, which satisfies the time-dependent integro-differential equation (also called nonlocal FPE), i.e.,

\begin{equation}
\label{FPE3}
  \frac{\partial p}{\partial t}=-\nabla \cdot\left(f(x)p(x,t)\right)+\mathcal{L}p(x,t), ~~p(x,0)=g(x).
\end{equation}

Superposition principle tells that the existence and uniqueness of the weak solution for equation \eqref{FPE3}.
To obtain  the existence and uniqueness of the strong solution, we need to establish the global Schauder estimate for the solution of \eqref{FPE3}.

\begin{theorem}(Schauder estimate.)\label{t11}
\label{schauder}
Under Hypothesis  $\bf{H.1}$ and Hypothesis  $\bf{H.2}$. Suppose that $\dive{f}\in C^{\beta}_b(\mathbb{R}^{d},\mathbb{R}), g\in C^{(1\vee \alpha) +\beta}_b(\mathbb{R}^{d},\mathbb{R}), $ and that there exists a positive constant $C$, such that  for the solution $p(x,t)$ of nonlocal FPE \eqref{FPE3} with $\alpha \in (0,1)\cup
(1,2)$, it holds that
\begin{equation}
\label{esti0}
\|p(x,t)\|_{L^\infty\l([0,T],C^{\a+\beta}_b\r)}\leq C\|g\|_{C^{(1\vee \alpha)+\beta}_b}.
\end{equation}
\end{theorem}
\begin{proof}
See Appendix I.
\end{proof}
Then by Theorem \ref{schauder}, we obtain the following result.
\begin{col}(Existence and uniqueness of the solution for the nonlocal FPE.)
 Let $\a\in (0,1)\cup
(1,2)$ be fixed, then there exists a unique solution $u\in C_b^{\a+\beta}([0,T],\mathbb{R}^d)$ to the nonlocal FPE
 (\ref{FPE3}), which also satisfies the Schauder estimate (\ref{esti0}).
\end{col}
\begin{proof}
Define a family of linear operators by  $T_\theta =(1-\theta)T_0+\theta T_1$, where
\begin{equation*}
\begin{split}
  T_0p_\tau&=p_\tau+\mathcal{L}^{*}p_\tau-f\nabla p_\tau,\\
  T_1p_\tau&=p_\tau+\mathcal{L}^{*}p_\tau-f\nabla p_\tau-\dive{f} \cdot  p_\tau.
\end{split}
  \end{equation*}
Note that $T_0$ is the case considered in \cite[Theorem 3]{DR}, which  maps $C_b^{\a+\beta}([0,T],\mathbb{R}^d)$ onto $C_b^{\beta}([0,T],\mathbb{R}^d)$. The solvability of the nonlocal FPE
 \eqref{FPE3} is equivalent to the invertibility of the operator $T_\theta$. By the method of continuity, we obtain the required result.
\end{proof}
{\bf{C: Numerical analysis}}

In this subsection, we devised a convergent finite difference method to simulate the nonlocal FPE \eqref{FPE3}.\\
{\bf{C1: A convergent finite difference scheme }}\\
Consider the problem in two cases: one case is on a bounded standard domain $D=(-1, 1)$ with the absorbing condition, and the other is on the infinite domain $\mathbb{R}$.  The absorbing condition is that the ``partical'' $X_t$ disappears or is killed when $X_t$ is outside a bounded domain $D$. So the probability density function $p(x,t)$
of being outside of the bounded domain $D$ is zero, i.e.,
\begin{equation}
    p(x,t) = 0, \;\;\; x \notin (-1,1).
\end{equation}
In fact,  by the following transformation
\begin{equation}
\begin{aligned}
\tilde{x}&=\frac{2x-(a+b)}{b-a}, ~~\tilde{y}=\frac{2y}{b-a},     \\
u(\tilde{x},t)&=p\left(\frac{(b-a)\tilde{x}+(a+b)}{2},t\right).
\end{aligned}
\end{equation}
We can also convert the finite interval $\widetilde{{D}}=(a,b)$ into standard domain ${D}=(-1,1)$. Then the nonlocal FPE  \eqref{FPE3} can be rewritten as
\begin{equation}
\label{v}
\begin{aligned}
&\frac{\partial u(\tilde{x},t)}{\partial t}=-\frac{\partial\left(\tilde{f}(\tilde{x}) u(\tilde{x},t)\right)}{\partial{{\tilde{x}}}}\\
&+ \int_{\mathbb{R}\setminus \{0\}}\left[u(\tilde{x}+\tilde{y},t)-u(\tilde{x},t)-I_{\{|\tilde{y}|<\frac{2}{b-a}\}}(\tilde{y})\tilde{y}\frac{\partial u(\tilde{x},t)}{\partial \tilde{x}}\right]
\frac{\widetilde{C}_{\alpha}}{e^{\tilde{\lambda}|\tilde{y}|}|\tilde{y}|^{1+\alpha}}d\tilde{y},  ~~\tilde{x} \in (-1,1),
\end{aligned}
\end{equation}
where
\begin{equation}
\begin{aligned}
\tilde{f}(\tilde{x})&=\frac{2}{b-a}f\left(\frac{(b-a)\tilde{x}+a+b}{2}\right), \\
 \widetilde{C}_{\alpha}&=(\frac{b-a}{2})^{-\alpha}C_\alpha, ~~\tilde{\lambda}=\frac{b-a}{2} \lambda.
\end{aligned}
\end{equation}
Next, the detailed procedure of the proposed  algorithm will
be described as follows:
\par
$(i)$ Step one is to introduce the following function,
\begin{equation}
 \int_s^{\infty} x^{-\varrho} e^{- x}\dx = s^{-\frac{\varrho}{2}}e^{-\frac{s}{2}}W_{-\frac{\varrho}{2},\frac{1-\varrho}{2}}(s), \text{for}\; s>0,
\end{equation}
where $W $ is the Whittaker W function.

The principal value of integral
   $\int_{\mathbb{R}\setminus \{0\}} 1_{\{|y|<1\}}(y)y \cdot \frac{\partial p(x,t)}{\partial x} \nu(\mathrm{d}y)$ vanishes. Thus the integral term  of equation \eqref{FPE3} becomes
\begin{small}
\begin{equation}
\label{FD}
\begin{aligned}
&\int_{\mathbb{R}\setminus \{0\}}\left[p(x+y,t)-p(x,t)\right]\nu(\mathrm{d}y)\\
&= C_{\a}\int_{-\infty}^{-1-x} \frac{p(x+y,t)-p(x,t)}{e^{\lam |y|}|y|^{1+\a}} \dy +C_{\a}\int_{-1-x}^{1-x} \frac{p(x+y,t)-p(x,t)}{e^{\lam |y|}|y|^{1+\a}} \dy+ C_{\a}\int_{1-x}^{\infty} \frac{p(x+y,t)-p(x,t)}{e^{\lam |y|}|y|^{1+\a}} \dy   \\
&= -C_{\a} p(x,t) \left[W_1(x)+W_2(x) \right]+C_{\a} \int_{-1-x}^{1-x} \frac{p(x+y,t)-p(x,t)}{e^{\lam |y|}|y|^{1+\a}} \dy,
\end{aligned}
\end{equation}
\end{small}
where
\begin{equation}
\begin{aligned}
    W_1(x)&= \lam^{\frac{\a-1}{2}}(1+x)^{-\frac{\a+1}{2}} e^{-\frac{\lam (1+x)}{2}} W_{-\frac{1+\a}{2},-\frac{\a}{2}}(\lam(1+x)),\\
      W_2(x)&= \lam^{\frac{\a-1}{2}}(1-x)^{-\frac{\a+1}{2}} e^{-\frac{\lam (1-x)}{2}} W_{-\frac{1+\a}{2},-\frac{\a}{2}}(\lam(1-x)).
\end{aligned}
\end{equation}

$(ii)$ Step two is to divide the interval $[-2,2]$ into $4J$ subintervals and define $x_j=jh$ for $-2J\leq j\leq 2J$ integer, where $h=\frac{1}{J}$. Using central difference scheme for the first and two derivatives and modifying the ``punched-hole" trapezoidal rule in the nonlocal term, we get the discretization scheme of \eqref{FPE3}, i.e.,
\begin{equation}
\begin{aligned}
\label{eq:DiscreteFPE}
   \frac{d{P_j}}{dt}:=& C_h\frac{P_{j+1}-2P_j+P_{j-1}}{h^2} -[(fP)_{x,j}^+ +(fP)_{x,j}^-]\\ 
     &- C_{\a}  \left[W_1+W_2 \right]P_j
        +  C_{\a} h \sum\limits_{k=-J-j, k\neq 0}^{J-j}\!{''}\frac{P_{j+k}-P_j}{e^{\lam |x_k|}|x_k|^{1+\a}},
\end{aligned}
\end{equation}
where $C_h = - C_{\a} \zeta(\a-1)h^{2-\a}$, $(fP)_{x,j}^+$ and $(fP)_{x,j}^-$ are defined as the global Lax-Friedrichs flux splitting \cite{Nsl},
i.e., $(fP)^{\pm} = \frac12 (fP \pm MP)$ with $M=\max{|f(x)|}$.  Here $\sum^{''}$ denotes  the quantities corresponding to the two end summation
indices are multiplied by $1/2$.

It  is noted that
for solving the initial value problem of \eqref{FPE3} on the infinite domain $\mathbb{R}$, the semi-discrete equation becomes
\begin{align}
\label{semidis}
  \frac{d{P_j}}{dt}:=& C_h\frac{P_{j+1}-2P_j+P_{j-1}}{h^2} -[(fP)_{x,j}^+ +(fP)_{x,j}^-]
        +  C_{\a} h \sum\limits_{k=-J-j, k\neq 0}^{J-j}\!{''}\frac{P_{j+k}-P_j}{e^{\lam |x_k|}|x_k|^{1+\a}},
\end{align}
where $J=\frac{\widetilde{L}}{h}$ and $\widetilde{L}$ is large enough so that the results are convergent.

($iii$) Step three is to give the condition of numerical scheme satisfying the discrete maximum principle. For the absorbing boundary condition and forward Euler scheme for time derivative, the scheme (\ref{eq:DiscreteFPE})
 satisfies the discrete maximum principle with $f=0$, if  $\bigtriangleup t$ and $h$ satisfy the following condition,
 \bear  \label{MPcondition}
    \frac{\bigtriangleup t }{h^\a}  \leq \frac{1}{2 C_{\a} [1+\frac{1}{\a}-\zeta(\a-1)] }.
 \enar


%


($iv$) Step four is to illustrate the  convergence analysis. The numerical solution $P^{n}_{j}$ of \eqref{semidis} converges to the analytic solution to \eqref{FPE3} for $x_j$ in $[-\tilde{L}/2, \tilde{L}/2]$ when the refinement path satisfies \eqref{MPcondition} and the length of the integration interval $2\tilde{L}$  in \eqref{semidis} tends to $\infty$.
\begin{remark}
  The detailed proofs of Step ($iii$) and Step ($iv$) see  Appendix II and III respectively.
\end{remark}
\noindent{\bf{C2: Numerical experiments}}\\
Here we present an example to illustrate our numerical method. We take the initial condition $p(x,0) = \sqrt{\frac{40}{\pi}}e^{-40x^2}$ and the finite interval $D=(-4,4)$. First, we consider the effect of  stability index $\alpha$ in Fig.\ref{difalp}. We
take $\alpha=0.2,0.6,1.2,1.6$ without drift coefficient $f=0$ at time
$t=0.5$.  As we see,
the larger the $\alpha$ becomes, the flatter the probability density function is.
Then, we illustrate the evolution for the probability density function with $\alpha=1.5, f=x-x^3$ at different
time $t=0.4, 0.8, 1.6$ and other factors fixed in Fig.\ref{difft}.  As time goes on,
the `particles' are gradually tend to stay at the two stable points. The maxima of the probability density function
approach $x=\pm 1$. Next, we present the effect of tempering parameter $\lambda$ in Fig.\ref{diflam}.
From the figure, we see that value of the density become larger near the origin when $\lambda$ is larger.
It is the opposite far away from the origin. We also present the effect of  drift terms ($f=0, x-x^3$) for probability density function with  $ \a=1.5, \lambda=0.01$  at time $t=1$. The particles centered at $x=0$ for $f=0$. While, for $f=x-x.^3$, the points $x=+-1$ are the two stable steady states. So the particles centered at one of the stable states.  Lastly, we add the Monte Carlo solutions
of stochastic differential equations to verify the correctness of the finite difference method  for $ \a=0.5, \lambda=0.01, f=0$  at time
$t=4$ in Fig.\ref{VerfyMC}.

\befig[h]
 \begin{center}
\includegraphics[width=0.8\linewidth]{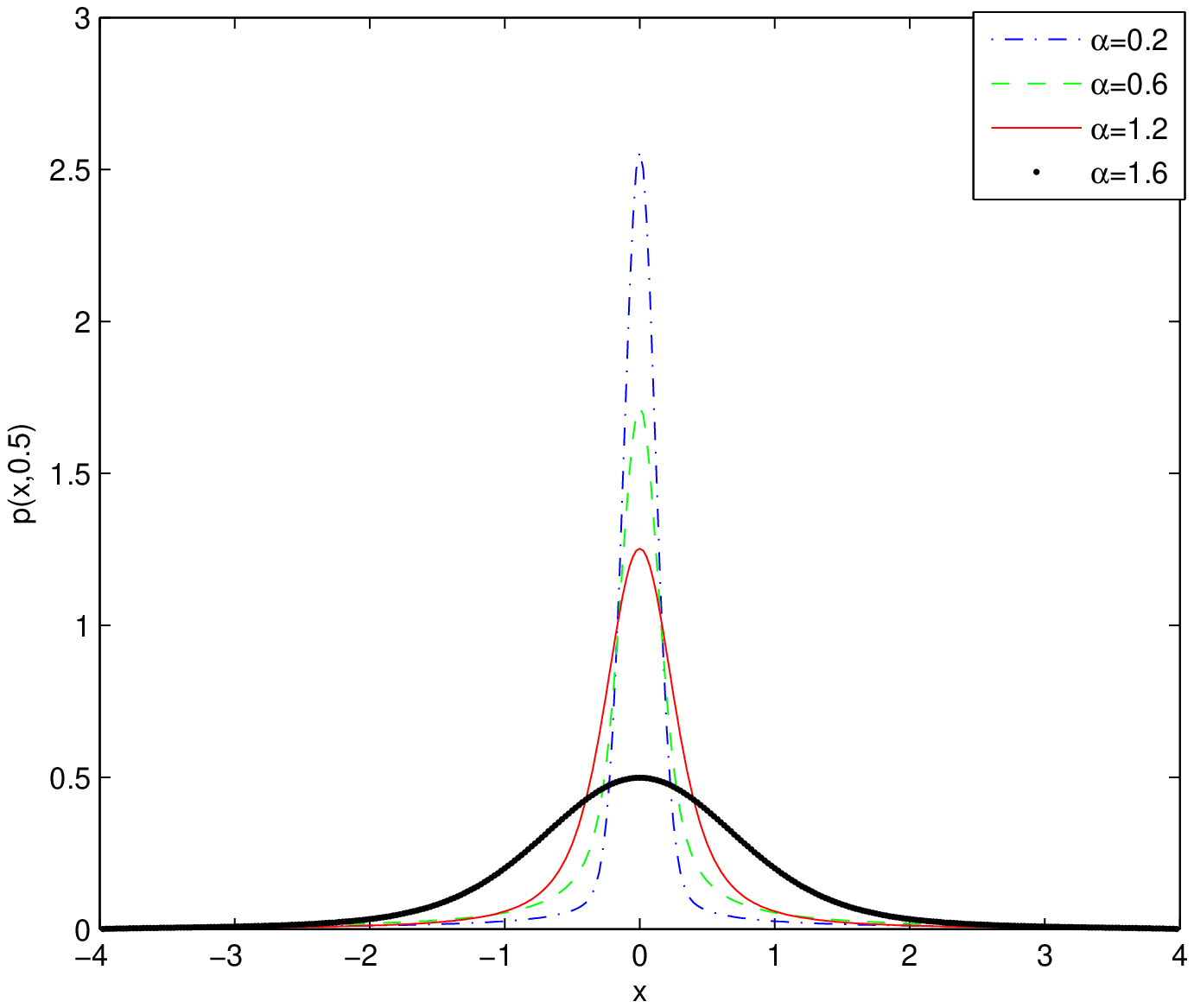}
\vspace{-0.3in}
\end{center}
\caption{ The effect of stability index $\alpha=0.2,0.6,1.2,1.6$ for probability density function without drift coefficient ($f=0$) and $\lambda=0.01$ at time $t=0.5$. }
\label{difalp}
\enfig

\befig[H]
 \begin{center}
\includegraphics[width=0.8\linewidth]{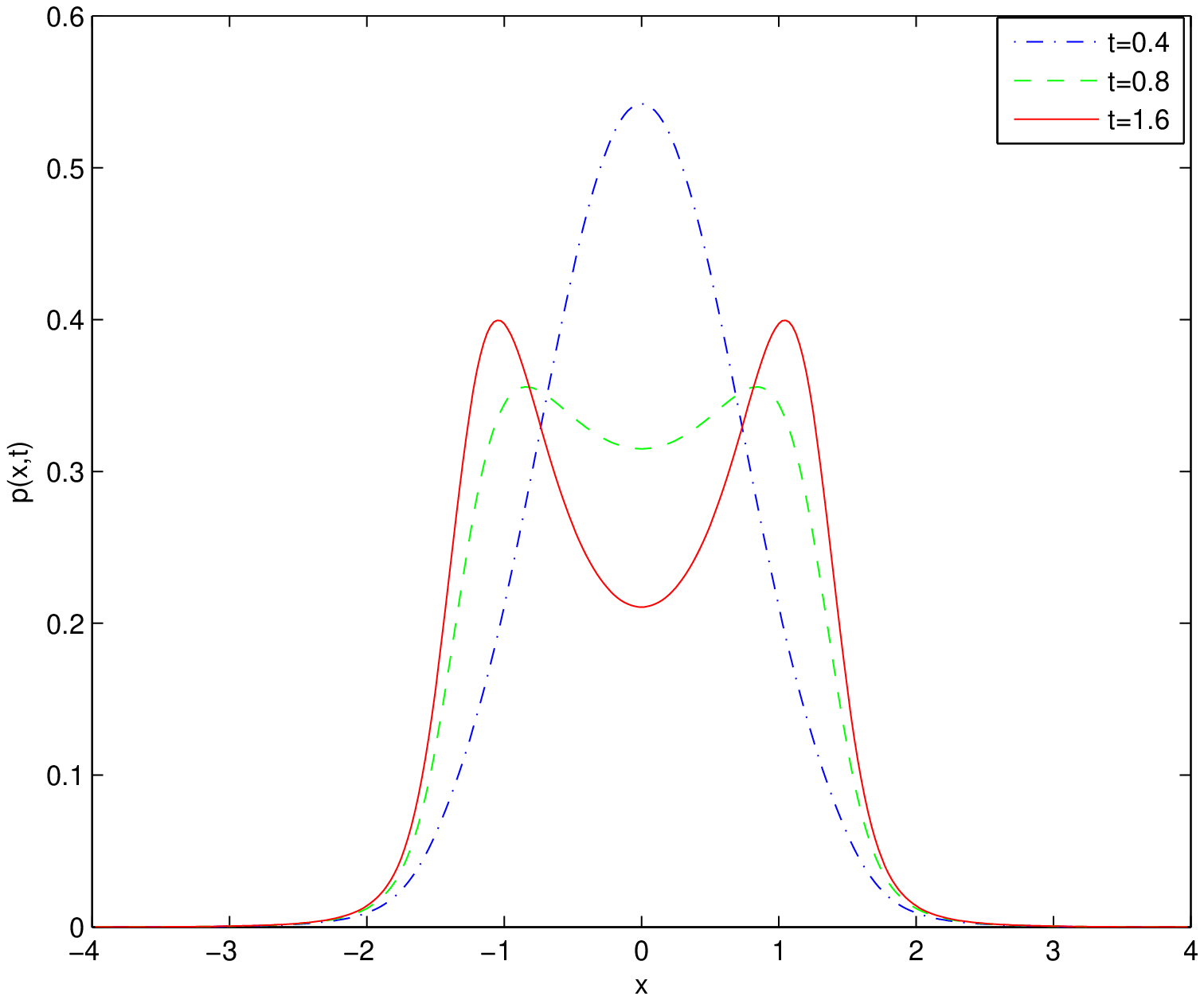}
\vspace{-0.3in}
\end{center}
\caption{ The evolution for probability density function with drift coefficient $f=x-x^3, \a=1.5, \lambda=0.01$  at time
$t=0.4,0.8,1.6$. }
\label{difft}
\enfig

\befig[H]
 \begin{center}
\includegraphics[width=0.8\linewidth]{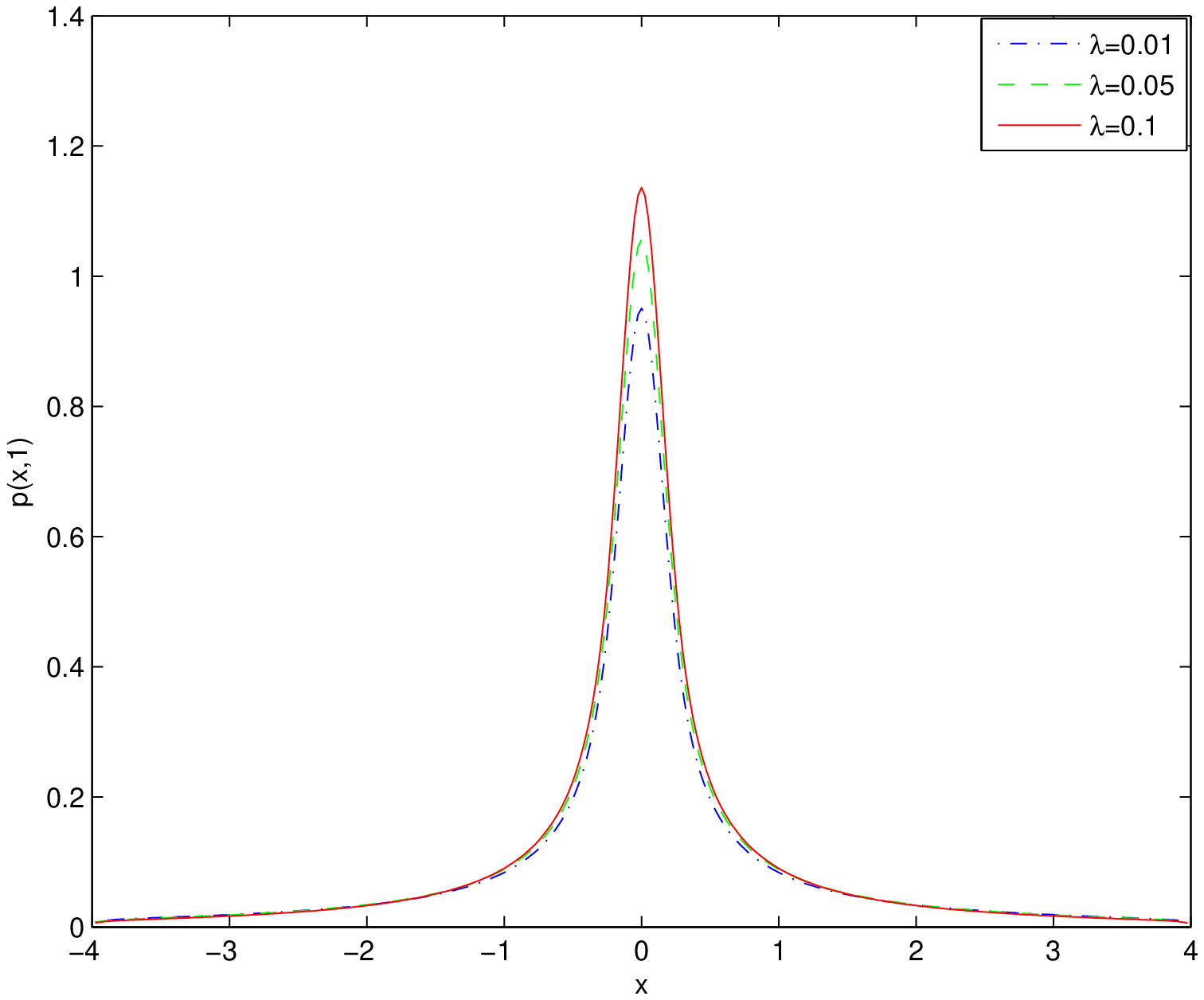}
\vspace{-0.3in}
\end{center}
\caption{ The effect of positive tempering parameter  $\lambda=0.01,0.05,0.1$ for probability density function without drift coefficient ($f=0$) at time $t=1$.  }
\label{diflam}
\enfig

\befig[H]
 \begin{center}
\includegraphics[width=0.8\linewidth]{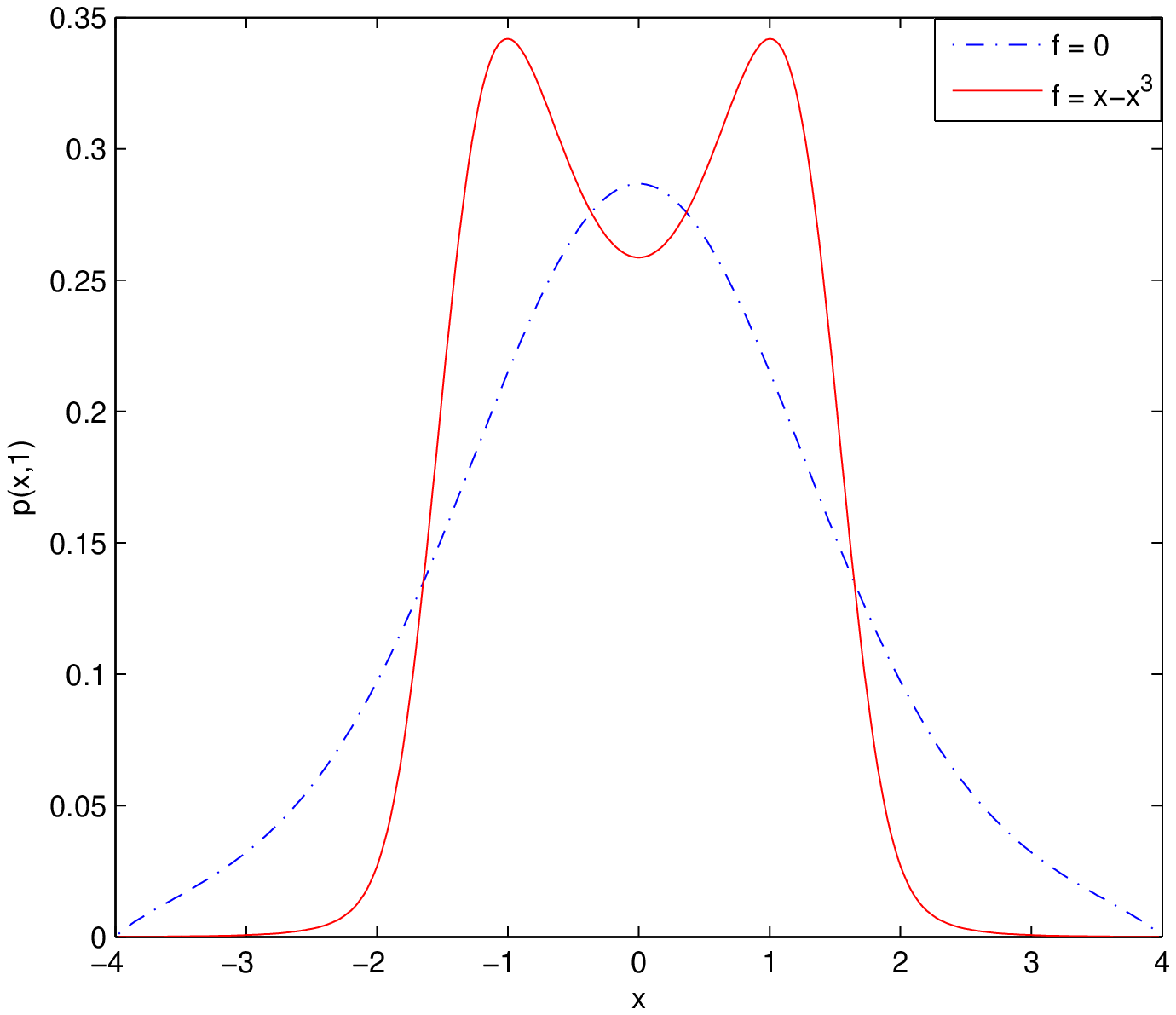}
\vspace{-0.3in}
\end{center}
\caption{ The effect of  drift terms ($f=0, x-x^3$) for probability density function with  $ \a=1.5, \lambda=0.01$  at time
$t=1$. }
\label{CompDrift}
\enfig

\befig[H]
 \begin{center}
\includegraphics[width=0.8\linewidth]{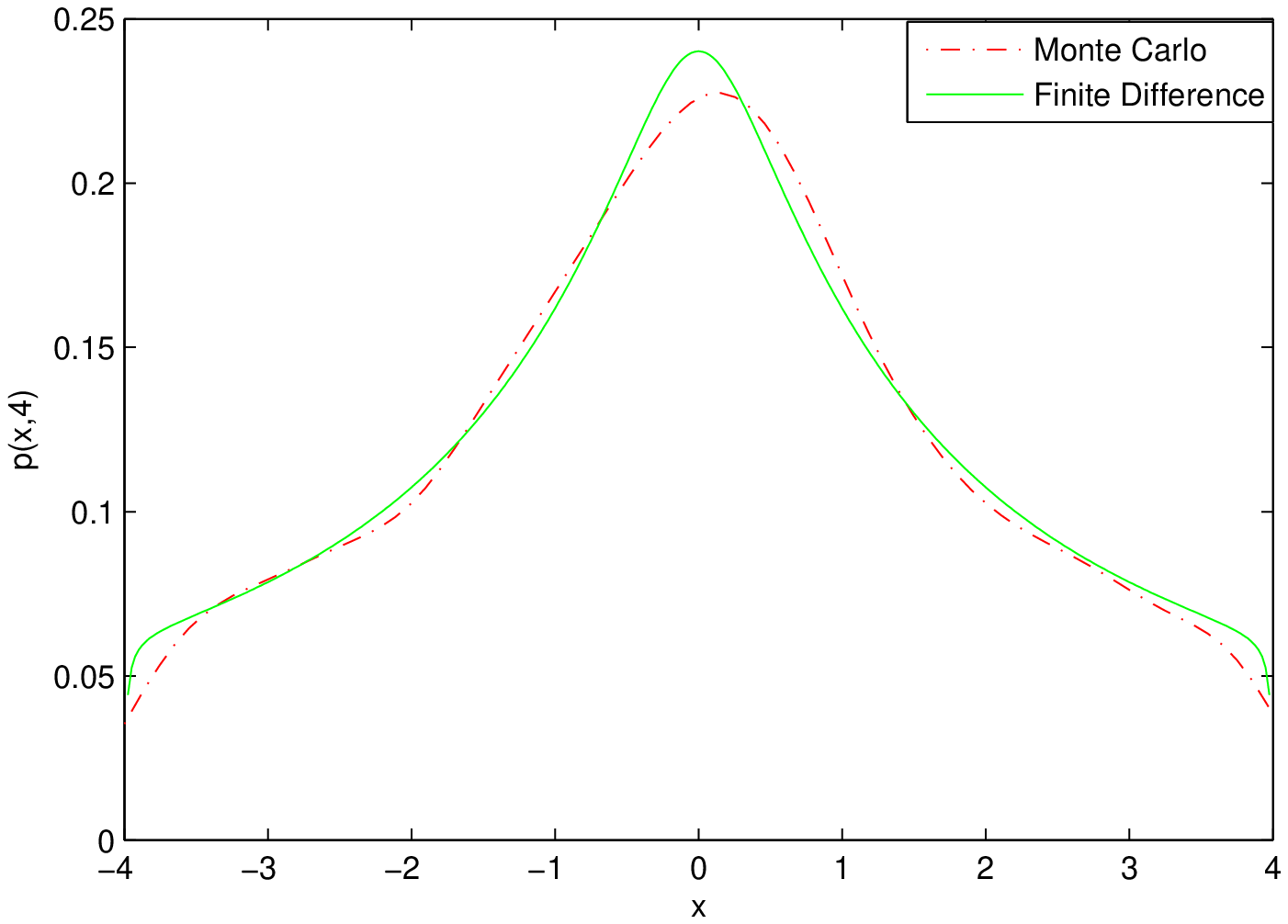}
\vspace{-0.3in}
\end{center}
\caption{ Compare the Monte Carlo simulation with finite difference method for $ \a=0.5, \lambda=0.01, f=0$  at time
$t=4$. }
\label{VerfyMC}
\enfig

\section{ Application to a nonlinear filtering problem}
The data assimilation method represented by nonlinear filtering has been widely used in many fields, which is a procedure to extract system state information by observations. Here we apply our aforementioned theoretical and numerical results to a nonlinear filtering system by simulating a nonlocal Zakai equation.

\subsection{The strong form of Zakai equation}

Consider the following signal-observation systems on $\mathbb{R}^2$
\begin{equation}
\label{so}
\left\{
\begin{aligned}
dX_t&=f(X_t)dt+dL_t, \\
dY_t&=h(X_t)dt+ d{W_t},
\end{aligned}
\right.
\end{equation}
where $f$ is given deterministic measurable function and $ h $ is a bounded measurable continuous function. $W_t$ is the standard Brownian motion, which is independent of $L_t$.

Let
\begin{equation}
\mathcal{Z}_t=\bf{{\sigma}}(Y_s:0\leq s\leq t)\vee \mathcal{N},
\end{equation}
where $\mathcal{N}$ is the  collection of all $\mathbb{P}$ -negligible sets of $(\Omega,\mathcal{F})$.
Define $\mathcal{Z}=\sigma(\bigcup_t\mathcal{Z}_t)$. By the version of Girsanov's change of measure theorem, we obtain a new probability measure $\widetilde {\mathbb{P}}$, such that the
observation $Z_t$ becomes $ \widetilde{\mathbb{P}}$-independent of the signal variables $(X_t,Y_t)$. This can be done through
\begin{equation}
\frac{d\widetilde {\mathbb{P}}}{d\mathbb{P}}= \exp\left(-\sum\limits_{i = 1}^m {\int_0^t {{h^i}({X_s})dW_s^i} }- \frac{1}{2}\sum\limits_{i = 1}^m {\int_0^t {{h^i}{{({X_s})}^2}} } ds\right).
\end{equation}
For every bounded differentiable function $\phi$, by the Kallianpur-Striebel formula, we have the following representation
\begin{equation}
\mathbb{E}[\phi(X_t,Y_t)|\mathcal{Z}_t]=\frac{\widetilde{\mathbb{E}}\left[\widetilde{\mathbb{R}}_t\phi(X_t, Y_t)|\mathcal{Z}\right]}{\widetilde{\mathbb{E}}\left[\widetilde{\mathbb{R}}_t|\mathcal{Z}\right]},
\end{equation}
where
\begin{equation}
\widetilde{\mathbb{R}}_t=\left.\frac{d\widetilde {\mathbb{P}}}{d\mathbb{P}}\right|_{\mathcal{Z}_t}.
\end{equation}

The unnormalized conditional distribution of $\varphi(X_t)$, given $Y_t$, is defined as
\begin{equation}
{{P}_t}(\varphi)=\widetilde{\mathbb{E}}\left[\widetilde{\mathbb{R}}_t\phi(X_t, Y_t)|\mathcal{Z}_t\right].
\end{equation}
Heuristically, if the unconditional distribution of the signal ${P}_t(\varphi)$ has a density $p(x,t)$ with respect to Lebesgue measure for all $t>0$, i.e.,
\begin{equation}
{P}_t(\varphi)=\int_{\mathbb{R}}\varphi(x)p(x,t)dx,
\end{equation}
then the unnormalized density $p(x,t)$ satisfies the following Zakai equation.
\begin{theorem}
\label{thden}
Under Hypotheses $\bf{H1}$-$\bf{H2}$ and $\bf{H3}$, the probability density function $p(x,t)$ satisfies the following Zakai equation, i.e.,
\begin{equation}
\label{dens1}
dp(x,t)={{A}}^{*}p(x,t)dt+ h(x)p(x,t)dY_t,
\end{equation}
where
\begin{equation}
{{A}}^{*}p(x,t)=-\nabla\cdot\left(f(x)p(x,t)\right)+\int_{{\mathbb{R}^d}\backslash \{ 0\}}\left[p(x+z,t)-p(x,t)\right]\nu(dz).
\end{equation}
\begin{proof}
The proof is similar to \cite [Theorem 5]{Zhang}.
\end{proof}
\end{theorem}
\subsection{ An example}

Here we consider the gradient system  as the signal system with time-independent bistable potential,  which describes the evolution of physical phenomena in ideal fluctuating environments.
\begin{equation}
\label{example12}
\begin{aligned}
dX_t&=-\frac{\partial V(X_t,t)}{\partial x} dt+dL_t,
\end{aligned}
\end{equation}
where $V(x,t)=-\frac{1}{2}x^2+\frac{1}{4}x^4$.

In this experiment, we assume that the observation system is given by
\begin{equation}
\label{example13}
dY_t=\cos(X_t)dt+dW_t,
\end{equation}

Using Theorem \ref{thden}, the strong form of Zakai equation for the signal-observation system \eqref{example12}-\eqref{example13} is
\begin{equation}
\label{exam3}
dp(x,t)={\widetilde{\mathcal{L}}}^{*}p(x,t)dt+\cos x\cdot p(x,t) dY_t,
\end{equation}
where
\begin{equation*}
\begin{aligned}
{\widetilde{\mathcal{L}}}^{*}p(x,t)=&\ -\frac{\partial\left(f(x)p(x,t)\right)}{\partial x}
+\int_{{\mathbb{R}}\backslash \{ 0\}}\left[p(x+z,t)-p(x,t)\right]\nu(dz).
\end{aligned}
\end{equation*}

In Fig.~\ref{sign_obs}, we simulate the signal-observation processes in Eq.~(\ref{example12}) and
(\ref{example13})(see \cite{Jum}). We take $\lambda=0.01, \alpha =1.5, f=x-x^3, X(0)=-1, Y(0)=-1 $ with the terminal time $t=1$.
The solution of Zakai equation in Eq.~(\ref{exam3}) with $ p(x,0) = \sqrt{\frac{40}{\pi}}e^{-40x^2}$ is illustrated at time $t=1$ in Fig.~\ref{zakaiTemper}.

\befig[h]
 \begin{center}
\includegraphics[width=0.8\linewidth]{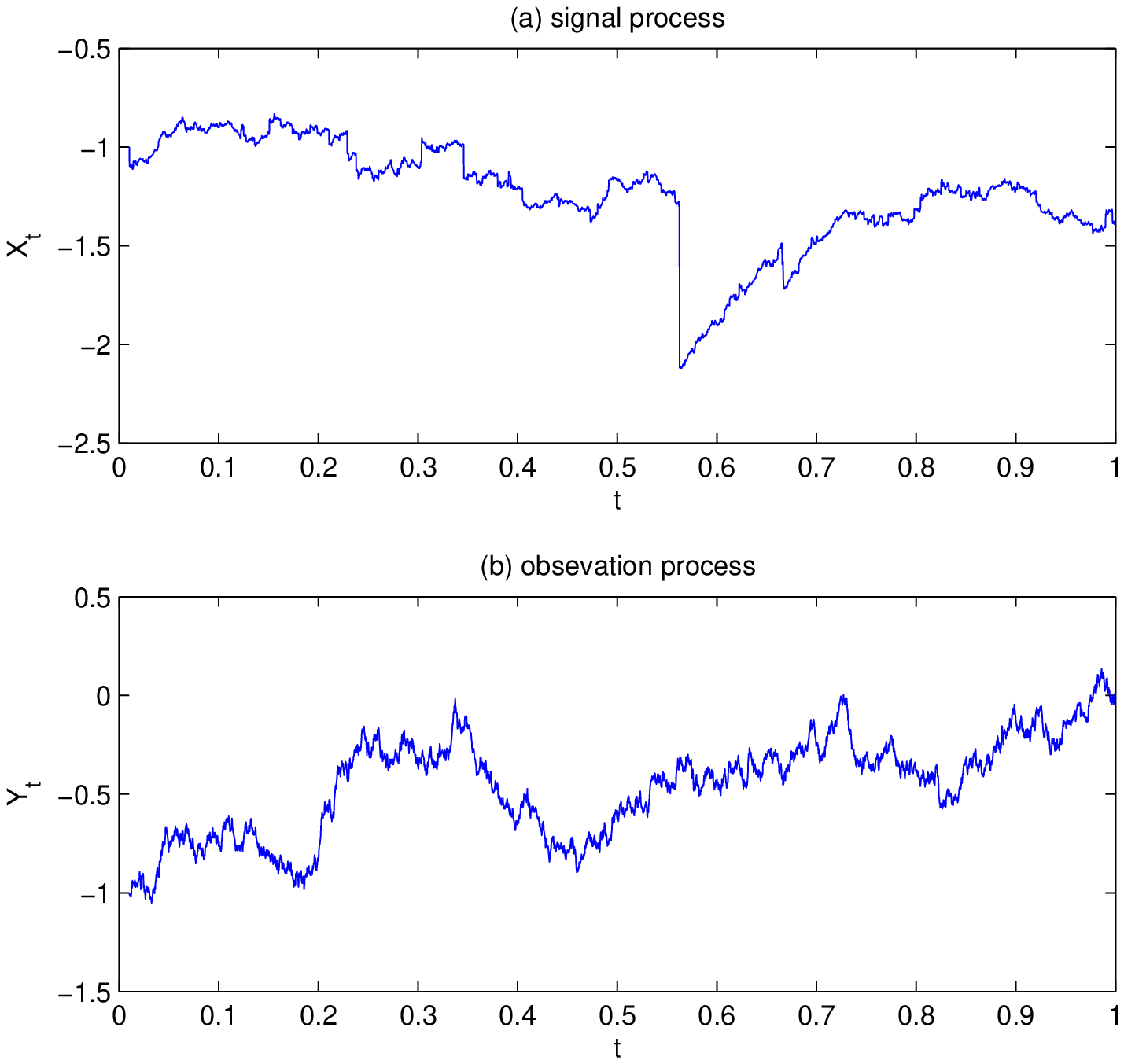}
\vspace{-0.3in}
\end{center}
\caption{ (a) The graph plots the signal process for $\lambda=0.01, \alpha =1.5, f=x-x^3, X(0)=-1, Y(0)=-1$. (b) The graph plots the observation process for $\lambda=0.01, \alpha =1.5, f=x-x^3, X(0)=-1, Y(0)=-1$. }
\label{sign_obs}
\enfig
\befig[H]
 \begin{center}
\includegraphics[width=0.8\linewidth]{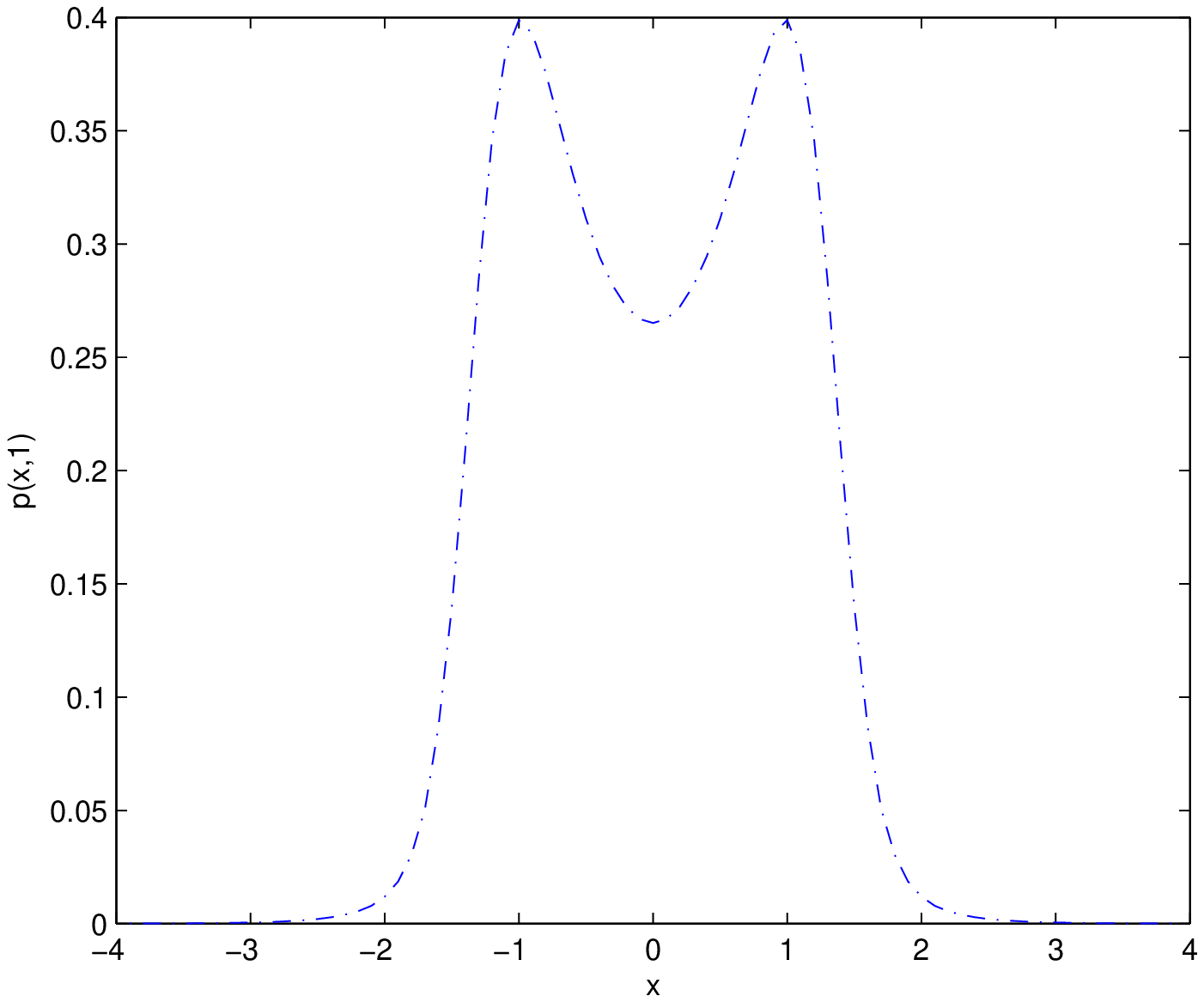}
\vspace{-0.3in}
\end{center}
\caption{ A sample path of Zakai equation with  $ p(x,0) = \sqrt{\frac{40}{\pi}}e^{-40x^2}$ at time
$t=1$.  }
\label{zakaiTemper}
\enfig

\section{Conclusion}
\par
In this paper, we have studied the superposition principle and numerical analysis for the nonlocal Fokker-Planck equation associated with a
stochastic dynamical system with tempered stable L\'{e}vy noise. Firstly, we have shown the superposition principle for probability measure-valued solution to a nonlocal Fokker Planck equation, which yields the equivalence between martingale problem for the stochastic dynamical system with tempered stable noise and  the corresponding nonlocal parabolic equation. Secondly, we have derived a  global parabolic Schauder estimate for the nonlocal Fokker-Planck equation under the condition $\alpha+\beta > 1$. This leads to the existence and uniqueness have been obtained in the $L^\infty\l(\left[0,T\right],C^{\a+\beta}_b\r)$. Thirdly, we have constructed a finite difference scheme to simulate the nonlocal Fokker-Planck equation. Under a specific condition, the scheme has been shown to satisfy the discrete maximum principle and to be convergent. Finally, we have applied the nonlocal Fokker-Planck equation to obtain the strong form of Zakai equation  and conduct its numerical assimilation. The results established in this paper can be used to examine dynamical behaviors for financial markets, climate dynamics and  physics.

\section*{Acknowledgements}
The research of L.Lin was supported  by the NSFC grants 11531006 and 11771449.  The research of X. Wang was supported by the NSFC grant 11901159. The research of Y. Zhang was supported by the NSFC grant 11901202.

\section*{Appendix}
\setcounter{equation}{0}
\setcounter{subsection}{0}
\renewcommand{\theequation}{A\arabic{equation}}
\renewcommand{\thesubsection}{A\arabic{subsection}}
\subsection*{Appendix ~I}
\begin{proof}
{ \bf{Step 1.}} For the case $\a\in(0,1),$ we set $\tau=T-t$. Then the equation \eqref{FPE3} can be rewritten as follows
 \begin{equation}\label{new1}
 \begin{cases}
   p_\tau-f\nabla p+\mathcal{L}^{*}p=\dive{f}\cdot p\\
   p(x, T)=g(x),
   \end{cases}
 \end{equation}
where $T>0$ is a fixed final time.

In the following, we will decompose the L\'{e}vy measure and give a bound for the character function.
Since  $\nu(dy)=e^{-\lambda|y|}\cdot\nu_{\a}(dy),$ where $\nu_{\a}$ is a symmetric stable L\'{e}vy measure. By the  polar coordinates $y=r\xi, (r,\xi)\in \mathbb{R_+}\times \mathbb{S}^{d-1}$, where $\mathbb{S}^{d-1}$ represents the $d-1$-dimensional sphere, then the jump measure $\nu_{\a}$  can be decomposed as
$$\nu_{\a}(B)=\int_{\mathbb{S}^{d-1}}\hat{\mu}(d\xi)\int_0^\infty 1_{B}(r\xi)\frac{dr}{r^{1+\a}},\; ~~\text{for} \;B\in \mathcal{B}{(\mathbb{R}^d)},$$
where $\hat{\mu}$ is a finite measure on $\mathbb{S}^{d-1}.$

 Note that
$\hat{\mu}$ is non-degenerate. Then there exists $\eta\geq1$, such that for all $m\in \mathbb{R}^{d},$
\begin{equation}\label{1231}
\eta^{-1}|m|^{\a}\leq\int_{\mathbb{S}^{d-1}}|(m,\xi)|^{\a}\hat{\mu}(d\xi)
\leq\eta|m|^{\a},~~ \a\in(0,1).
\end{equation}
The L\'{e}vy symbol associated with $\mathcal{L}$ is given
\begin{equation*}
\label{ine}
\begin{split}
\psi(\lambda)&=\exp{\left\{\int_{\mathbb{R}^{d}}\left(e^{(im, y)}-1\right)\cdot e^{-\lambda|y|}\nu_{\a}(dy)\right\}}\\
&=\exp{\left\{\int_{\mathbb{S}^{d-1}}\hat{\mu}(d\xi)\int_0^\infty(e^{i(m, r\xi)}-1)e^{-\lambda|r\xi|}\frac{dr}{r^{1+\a}}\right\}}\\
&=\exp{\left\{\int_{\mathbb{S}^{d-1}}\hat{\mu}(d\xi)\int_0^\infty\l[\l(e^{(i(m, \xi)-\lambda)r}-1\r)-\l(e^{-\lambda r}-1\r)\r]\frac{dr}{r^{1+\a}}\right\}}\\
&:=\exp\{I_1+I_2\}.
\end{split}
\end{equation*}
Set $\hat{\mu}_1=-\Gamma(-\a)\hat{\mu}$. Then we have
\begin{align}
I_1&=\int_{\mathbb{S}^{d-1}}\hat{\mu}(d\xi)\int_0^\infty\l(e^{(i(m ,\xi)-\lambda)r}-1\r)\frac{dr}{r^{1+\a}},\\
&=-\int_{\mathbb{S}^{d-1}}(\lambda-i(m,\xi))^{\a}\hat{\mu}_1(d\xi)\\
&=-\int_{\mathbb{S}^{d-1}}|\lambda-i(m,\xi)|^{\a}e^{i\a \cdot arg(\lambda-i(m,\xi))}\hat{\mu}_1(d\xi),
\end{align}
and
\begin{equation}
\label{0I2}
\begin{aligned}
I_2&=\int_{\mathbb{S}^{d-1}}\hat{\mu}(d\xi)\int_0^\infty\l(e^{-\lambda r}-1\r)\frac{dr}{r^{1+\a}}
&=\lambda^{\a}\hat{\mu}_1(\mathbb{S}^{d-1}).
\end{aligned}
\end{equation}
Define $\theta=\arg(\lambda-i(m,\xi))$. Then we have
$$e^{i\a \cdot arg(\lambda-im\xi)}=e^{i\a\theta}=\cos(\a\theta)+i\sin(\a\theta).$$
Since  $\hat{\mu}_1$ is a symmetric measure, we conclude that
$$I_1=-\int_{\mathbb{S}^{d-1}}\l((m,\xi)^2+\lambda^2\r)^{\a/2}cos(\a\theta)\hat{\mu}_1(d\xi).$$
Recall that $I_1$ is the L\'{e}vy symbol, $\mathscr{R}I_1\leq0$. There exists $\eta>1,$ such that
\begin{equation}
\label{0I1}
-\eta |m|^{\a}-C\leq I_1\leq -\eta^{-1}|m|^{\a},
\end{equation}
where $C$ is a positive constant.

Combined the inequality \eqref{0I1} with equality \eqref{0I2}, we have
\begin{equation}
\label{i12}
-\eta|m|^{\alpha}+\lambda^{\alpha}\hat{\mu}_1(\mathbb{S}^{d-1})-C\leq I_1+I_2 \leq \lambda^{\alpha}\hat{\mu}_1(\mathbb{S}^{d-1})-\eta^{-1}|m|^{\alpha}
\end{equation}

{ \bf{Step 2.}} We will establish the  upper bound for the heat-kernel and for the derivatives of the heat-kernel. In fact,
for the operator $\mathcal{L}$, it is well-known that the
associated convolution Markov semigroup has a $C^\infty$ density
$p_{\a}(\cdot,t)$. By Fourier inversion,  for all $t>0,$ $y\in \mathbb{R}^d$, we get
$$p_{\a}(y,t)=\frac{1}{(2\pi)^d}\int_{\mathbb{R}^d}\exp(-i(y,m))\exp\left(t(I_1+I_2\right)dm.$$
Moreover, we have the following two results :
\begin{enumerate}
 \item For some $c=c(\gamma,\alpha)>0$ and $t\in(0,1]$, there exists $\a\in (0,1)$ such that for every $\gamma$ satisfying $0<\gamma<\a$, we have
  $\int_{\mathbb{R}^d}|y|^\gamma p_{\a}(y,t)dy\leq ct^{\gamma/{\a}}.$
  \item There exists a positive constant $c=c(\a,\beta)>0$ such that
  $\int_{\mathbb{R}^d}|y|^\beta |D_y^kp_{\a}(y,t)|dy\leq \frac{c}{t^{[k-\beta]/{\a}}},\; t\in (0,1], k=1,2.$
  Here $D_y^1p_{\a}(y,t)$ and $D_y^2p_{\a}(y,t)$
   denote the gradient and Hessian matrix in the
  $y$-variable, respectively.
\end{enumerate}
 On the one hand, by the inequality \eqref{i12}, we  have
\begin{equation*}
  e^{(I_1+I_2)}\leq e^{\lambda^{\a}\hat{\mu}_1(\mathbb{S}^{d-1})}e^{-\eta^{-1}|m|^{\a}}\leq C e^{-\eta^{-1}|m|^{\a}}.
\end{equation*}
Thus
\begin{equation*}
\begin{split}
  \int_{\mathbb{R}^d}|y|^\gamma p_{\a}(y,t)dy&\leq C\int_{\mathbb{R}^d}\int_{\mathbb{R}^d}e^{-i(y,m)}|y|^\gamma e^{-t\eta^{-1}|m|^{\a}}dm dy\\
  &\leq C\int_{\mathbb{R}^d}\l(\int_{\mathbb{R}^d}
 e^{-i(y,m)}|y|^\gamma dy\r) e^{-t\eta^{-1}|m|^{\a}}dm\\
 &\leq C \int_{\mathbb{R}^d}|m|^{-\gamma-1}e^{-t\eta^{-1}|m|^{\a}}dm\\&\leq C \int_{\mathbb{R}^d}e^{-t\eta^{-1}|m|^{\a}}
 (t\eta^{-1}|m|^{\a})^{\frac{-\gamma-\a}{\a}}t^{\frac{\gamma+\a}{\a}}\frac{1}{t}d(-t\eta^{-1}|m|^{\a})\\
 &\leq C t^{\frac{\gamma}{\a}}.
  \end{split}
\end{equation*}
On the other hand,
\begin{equation*}
\begin{split}
\int_{\mathbb{R}^d}|y|^\beta |D_y^kp_{\a}(y,t)|dy&\leq \int_{\mathbb{R}^d}\int_{\mathbb{R}^d}|m|^ke^{-i(y,m)}|y|^\beta e^{-t\eta^{-1}|m|^{\a}}dm dy\\
&\leq C \int_{\mathbb{R}^d}|m|^{k-\gamma-1}e^{-t\eta^{-1}|m|^{\a}}dm\\
&\leq Ct^{\frac{\beta-k}{\a}}\leq \frac{C}{t^{[k-\beta]/{\a}}}.
  \end{split}
\end{equation*}
{ \bf{Step 3.}}
Since $\dive{f}\in C^{\beta}_b$, which will not affect the discuss  as in \cite[Theorem 3]{DR}. We also have
\begin{equation}
\label{esti}
\|p(x,t)\|_{L^\infty\l([0,T],C^{\a+\beta}_b\r)}\leq C\|g\|_{C^{1+\beta}_b}.
\end{equation}
{ \bf{Step 4.}} For the case $ \alpha \in (1,2),$ the condition $\alpha+\beta>1$ is obviously satisfied. Further, the characteristic function is the same as the case of $ \alpha \in (0,1).$ We can obtain the upper bound for the heat-kernel and the derivatives of the
heat-kernel in the same way as $ \alpha \in (0,1)$, i.e., 
\begin{equation}
\label{esti12}
\|p(x,t)\|_{L^\infty\l([0,T],C^{\a+\beta}_b\r)}\leq C\|g\|_{C^{\alpha+\beta}_b}.
\end{equation}
Therefore, for the case $\alpha \in (0,1)\cup(1,2)$, we have 
\begin{equation}
\label{esti01}
\|p(x,t)\|_{L^\infty\l([0,T],C^{\a+\beta}_b\r)}\leq C\|g\|_{C^{(1\vee\alpha)+\beta}_b}.
\end{equation}
Moreover, the global Schauder estimates still holds.
\end{proof}

\subsection*{Appendix ~II}
\setcounter{equation}{0}
\renewcommand\theequation{B\arabic{equation}}
\begin{proof}
   Let $M$ be the maximum value of the initial probability density, then for the numerical solution $0<P_j^n \leq M$, we have $\zeta(\a-1)\leq 0$. Applying the explicit Euler to the semi-discrete scheme \eqref{eq:DiscreteFPE}, we get
    \bess
    \label{depo}
        P_j^{n+1} &=& P_j^n-\delt C_{\a} \zeta(\a-1)h^{-\a} (P_{j+1}^n-2P_j^n+P_{j-1}^n) \\
            &&- C_{\a}\delt  \left[W_1(x_j)+W_2(x_j) \right]P_j^n+  C_{\a}\delt h \sum\limits_{k=-J-j, k\neq 0}^{J-j}\!{''}\frac{P_{j+k}^n-P_j^n}{e^{\lam |x_k|}|x_k|^{1+\a}} \\
            &=&\bigg[1+2\delt C_{\a}\zeta(\a-1) h^{-\a}- C_{\a}(W_1(x_j)+W_2(x_j))\delt \\
            &&-  C_{\a}\delt h \sum\limits_{k=-J-j, k\neq 0}^{J-j}\!{''}\frac{1}{e^{\lam |x_k|}|x_k|^{1+\a}} \bigg] P_j^n
             -\delt C_{\a}\zeta(\a-1)h^{-\a}(P_{j-1}^n+P_{j+1}^{n})   \\
            &&+  C_{\a}\delt h \sum\limits_{k=-J-j, k\neq 0}^{J-j}\!{''}\frac{P_{j+k}^n}{e^{\lam |x_k|}|x_k|^{1+\a}}.
 \eess
Set
\begin{equation}
\label{deo}
\begin{aligned}
    L:&= C_{\a}(W_1(x_j)+W_2(x_j))\delt  + C_{\a}\delt h \sum\limits_{k=-J-j, k\neq 0}^{J-j}\!{''}\frac{1}{e^{\lam |x_k|}|x_k|^{1+\a}}.
    \end{aligned}
\end{equation}
 Then by the inequality (\ref{MPcondition}), we have
 \begin{equation}
\begin{aligned}
    L &\leq   C_{\a} \delt \left[\int_{-\infty}^{-1-x_j} \frac{\dy}{e^{\lam y}|y|^{1+\a} } +  \int_{1-x_j}^{\infty} \frac{\dy}{e^{\lam y}|y|^{1+\a} }
    +\frac{2h}{e^{\lam h}h^{1+\a}}+\int_{(-1-x_j+\frac{h}{2},1-x_j-\frac{h}{2})\backslash (-h,h) } \frac{\dy}{e^{\lam y}|y|^{1+\a} } \right]  \\
    &\leq  C_{\a} \delt \left[\frac{2}{e^{\lam h}h^\a}+2\int_h^\infty \frac{\dy}{e^{\lam y}|y|^{1+\a} }\right]\\
    &\leq \frac{2 C_{\a} \delt }{h^\a}(1+\frac{1}{\a})\\
    &\leq 1+2\delt C_{\a}\zeta(\a-1) h^{-\a},
\end{aligned}
\end{equation}
where $x_j=jh$.

Therefore we have
\begin{equation}
 P_j^{n+1}\leq \left\{1- C_{\a}\delt[W_1(x_j)+W_2(x_j)]\right\}M \leq M.
\end{equation}
\end{proof}
\subsection*{Appendix ~III}
\setcounter{equation}{0}
\renewcommand\theequation{C\arabic{equation}}
\begin{proof}
 Set $e^{n}_{j}=P^{n}_{j}-p(x_j,t_n)$, where $p(x_j,t_n)$ is the analytic solution at the point $(x_j, t_n)$, then we have
\begin{equation}
\begin{aligned}
e^{n+1}_{j}-e^{n}_{j}&=-\Delta tC_{\alpha}\zeta(\alpha-1)h^{2-\alpha}\frac{e^{n}_{j-1}-2e^{n}_{j}+e^{n}_{j+1}}{h^2} \\
 &+\Delta tC_{\alpha}h\sum^{J-j}_{k=-J-j, k\neq 0}\frac{e_{j+k}^n-e_j^n}{e^{\lambda|x_k|}|x_k|^{1+\alpha}}
-\Delta t T^{n}_{j},
\end{aligned}
\end{equation}
where

\begin{equation}
\begin{aligned}
&T^{n}_{j}=C_{\alpha}\zeta(\alpha-1)h^{2-\alpha}\frac{p(x_{j-1},t_n)-2p(x_j,t_n)+p(x_{j+1},t_n)}{h^2}\\
&-C_{\alpha}h\sum^{J-j}_{k=-J-j, k\neq 0}\frac{p(x_{j+k},t_n)-p(x_j,t_n)}{e^{\lambda|x_k|}|x_k|^{1+\alpha}}+\frac{p(x_j,t_{n+1})-p(x_j,t_n)}{\Delta t}.
\end{aligned}
\end{equation}
Further, we have
\begin{equation}
\begin{aligned}
&T^{n}_{j} = \frac{1}{2}\frac{\partial^2p(x_j,t_n)}{\partial^2 t}\Delta t+A_{\alpha}\zeta(\alpha-1)h^{4-\alpha}\frac{\partial^4 p(x_j,t_n)}{\partial^4 x}-B_{\alpha}\zeta(\alpha-3)h^{4-\alpha}\frac{\partial^4 p(x_j,t_n)}{\partial^4 x}\\
&-D_{\alpha}\frac{\partial}{\partial y}\left(\frac{p(x_j+y,t_n)-p(x_j, t_n)}{e^{\lambda |y|}|y|^{1+\alpha}}\right)\Big|^{y=L-x_j}_{y=-L-x_j}\\
&+E_{\alpha}\int_{\{-\infty,-L-x_j\}\bigcup \{L-x_j, \infty\}}\frac{p(x_j+y,t_n)-p(x_j,t_n)}{e^{\lambda |y|}|y|^{1+\alpha}}dy+\cdots,
\end{aligned}
\end{equation}
where $ \zeta(\tau)$ is the Riemann zeta function initially defined for $\mathbb{R}e \tau > 1$ by $\zeta(\tau)=\sum_{k=1}^{\infty}k^{-\tau}$, $A_{\alpha}, B_{\alpha}, D_{\alpha},  E_{\alpha}$ are constants depending on $\alpha$.

Obviously, we have
\begin{equation}
|T^{n}_j|\leq O(\Delta t)+ O(h^2)+O(L^{-\alpha}):=\widetilde{{T}}.
\end{equation}
Therefore, the truncation error is uniformly bounded. By the condition \eqref{MPcondition}, we have
\begin{equation}
\max |e^{n+1}_j|\leq \max |e^{n}_j|+\Delta  t \widetilde{T}\leq n\Delta  t \widetilde{{T}}.
\end{equation}
\end{proof}

\section*{Data Availability Statements}

The data that support the findings of this study are openly available in GitHub, Ref. \cite{Lin}.

\end{document}